\documentclass[a4paper,12pt,intlimits,oneside]{amsart}

\usepackage{enumerate}
\usepackage{amsfonts}
\usepackage{amsmath,amsthm,amssymb}

\newcommand{\comment}[1]{}

\newtheorem{theorem}{Theorem}
\newtheorem{definition}[theorem]{Definition}
\newtheorem{proposition}[theorem]{Proposition}
\newtheorem{lemma}[theorem]{Lemma}

\newtheorem{problem}[theorem]{Problem}

\newtheorem{remark}[theorem]{Remark}

\newcommand\e{\varepsilon}
\newcommand{\card}{{\mbox{Card }}}

\newcommand{\NN}{\mathbb N}
\newcommand{\ZZ}{\mathbb Z}
\newcommand{\RR}{\mathbb R}
\newcommand{\CC}{\mathbb C}
\newcommand{\TT}{\mathbb T}

\newcommand{\GG}{\mathbb{G}_q}
\newcommand{\GS}{\mathbb{G}_q^{\star}}

\newcommand\PP{\mathcal{P}}

\newcommand{\eref}[1]{{\rm (\ref{\sharp1})}}
\newcommand{\Fpart}[1]{{\left\{{\sharp1}\right\}}}
\newcommand{\Abs}[1]{{\left|{\sharp1}\right|}}
\newcommand{\Lone}[1]{{\left\|{\sharp1}\right\|_{L^1}}}
\newcommand{\Linf}[1]{{\left\|{\sharp1}\right\|_\infty}}
\newcommand{\Norm}[1]{{\left\|{\sharp1}\right\|}}

\newcommand{\Floor}[1]{{\left\lfloor{\sharp1}\right\rfloor}}
\newcommand{\Ceil}[1]{{\left\lceil{\sharp1}\right\rceil}}

\begin{document}

\title[Integral concentration of idempotents]{Concentration of the integral norm of idempotents}

\author{Aline Bonami \& Szil\'{a}rd Gy. R\'{e}v\'{e}sz}

\date{\today}

\address[Aline Bonami]{
\newline\indent F\'ed\'eration Denis Poisson
\newline\indent MAPMO-UMR 6628 CNRS
\newline\indent Universit\'e d'Orl\'eans
\newline\indent 45067 Orl\'eans France.}
\email{aline.bonami@univ-orleans.fr}

\address[Szil\'ard Gy. R\'ev\'esz]{
\newline \indent A. R\'enyi Institute of Mathematics
\newline \indent Hungarian Academy of Sciences,
\newline \indent Budapest, P.O.B. 127, 1364 Hungary.}
\email{revesz@renyi.hu}

\comment{
\address{\hskip 1cm and
\newline \indent Institut Henri Poincar\'e,
\newline \indent 11 rue Pierre et Marie Curie, 75005 Paris, France}
\email{revesz@ihp.jussieu.fr}
}

\begin{abstract}
This  is a companion paper of a recent one, entitled {\sl Integral
concentration of idempotent trigonometric polynomials with gaps}.
New results of the present work concern $L^1$ concentration, while
the above mentioned paper deals with $L^p$-concentration.

Our aim here is two-fold. At the first place we try to explain
methods and results, and give further straightforward corollaries.
On the other hand, we push forward the methods to obtain a better
constant for the possible concentration (in $L^1$ norm) of an
idempotent on an arbitrary symmetric measurable set of positive
measure. We prove a rather high level $\gamma_1>0.96$, which
contradicts strongly the conjecture of Anderson et al. that there
is no positive concentration in $L^1$ norm.

The same problem is considered on the group
$\mathbb{Z}/q\mathbb{Z}$, with $q$ say a prime number. There, the
property of absolute integral concentration of idempotent
polynomials fails, which is in a way a positive answer to the
conjecture mentioned above. Our proof uses recent results of B.
Green and S. Konyagin on the Littlewood Problem.
\end{abstract}

\maketitle
\let\oldfootnote\thefootnote
\def\thefootnote{}

\footnotetext{The second author was supported in part by the
Hungarian National Foundation for Scientific Research, Project \#s
T-049301 K-61908 and K-72731, and also by the European Research
Council, Project \# ERC-AdG No. 228005.}

\section{Introduction and statement of results}\label{sec:intro}

 The problem of $p$-concentration
on the torus for idempotent polynomials has been considered first
in \cite{first}, \cite{CRMany}, \cite{CRSome}, \cite{DPQ}. We use
the notation $\TT:=\RR/\ZZ$ for the torus. Then $e(t):=e^{2\pi i
t}$ is the usual exponential function adjusted to interval length
$1$, and we denote $e_h$ the function $e(ht)$. For obvious reasons
of being convolution idempotents, the set
\begin{equation}\label{eq:idempotents}
\PP:=\left\{ \sum_{h\in H}e_h ~:~ H\subset \NN, ~ \sharp H< \infty
\right\}
\end{equation}
is called the set of \emph{(convolution-)idempotent exponential
(or trigonometric) polynomials}, or just \emph{idempotents} for
short. The $p$-concentration problem comes from the following
definition.

\begin{definition}
Let $p>0$. We say that there is  $p$-concentration  if  there
exists a constant $\gamma>0$ so that for any symmetric (with
respect to $0$) measurable set $E$ of positive measure one can
find an idempotent $f\in\PP$ with
\begin{equation}\label{eq:Lpconcentration}
\int_E |f|^p \geq \gamma\int_\TT |f|^p.
\end{equation}
The supremum of all such constants $\gamma$ will be denoted as
$\gamma_p$, and called the level of $p$-concentration.
\end{definition}

The main theorem of \cite{Many} can be stated as:
\begin{theorem}[{\bf Anderson, Ash, Jones, Rider, Saffari}]
\label{th:largepconcentration} There is $p$-concentration for all
$p>1$.
\end{theorem}

We prove in our recent paper \cite{BR} that there is
$p$-concentration for all $p>1/2$, while the same authors
conjectured that idempotent concentration fails already for $p=1$.
Moreover, we prove that the constant $\gamma_p$ is equal to $1$
when $p>1$ and $p$ is not an even integer. This is in line with
the fact that $L^p$ norms behave differently depending on whether
$p$ is an even integer or not in a certain number of problems,
such as the Hardy-Littlewood majorant problem ({\sl does an
inequality on absolute values of Fourier coefficients imply an
inequality on $L^p$ norms?}) or the Wiener property for periodic
positive definite functions ({\sl does a positive definite
function\comment{with large gaps in its Fourier series ???} belong
to $L^p$ when it is the case on a small interval?}). The fact that
one can find idempotents among counter-examples to the
Hardy-Littlewood majorant problem had been conjectured by
Montgomery \cite{M} and was recently proved by Mockenhaupt and
Schlag \cite{MS}, and we rely on their construction in \cite{BR}.
At the same time, we were able to revisit the Wiener property in
order to construct counter-examples among idempotents \cite{BR2}.

\medskip

Even if we disproved the conjecture of \cite{Many} for $p=1$, the
situation is not yet entirely clear. Indeed, the constant $\gamma$
can be taken arbitrarily close to $1$ when we restrict the class
of symmetric measurable sets to symmetric open sets or enlarge the
class of trigonometrical polynomials to all positive definite
ones, that is, allow all non negative coefficients and not only
$0$ or $1$. So one may conjecture that $\gamma_1=1$ (even if we
understand that one should be cautious with such conjectures). By
pushing forward our techniques, we improve our previous constant
and prove the following.

\begin{theorem}\label{th:L1} For $p=1$
there is concentration at the level $\gamma_1>0.96$.
Moreover, for arbitrarily large given $N$ the corresponding
concentrating idempotent can be chosen with gaps at least $N$
between consecutive frequencies.
\end{theorem}

In order to prove this theorem, we will describe the main steps of
our proofs in \cite{BR} before focusing on the improvements. When
doing this, we also give a relatively simple proof of the fact
that the best constant $\gamma_2$ for symmetric measurable sets is
the same as for open sets. This is proved in \cite{Many}, as it is
a particular case of their general result, but their proof is not
easy to read. We describe it here so that a simpler, explanatory
proof be available. The constant for open sets has been obtained
by D\'echamps-Gondim, Piquard-Lust and Queff\'elec \cite{DPQ,
DPQ2}, so that
\begin{equation}\label{case2}
\gamma_2=\sup_{0\leq x}\frac {2\sin^2 x}{\pi x}=0.4613\cdots.
\end{equation}
\bigskip

In all proofs, the same kind of estimates as
\eqref{eq:Lpconcentration}, but with finite sums on a grid of
points replacing integrals, plays a central role in the proofs. So
it was natural to get interested in best constants on these finite
structures. This led us to the same problem, but taken on finite
groups, which we describe now.

\bigskip

Let us consider $\ZZ_q:=\ZZ/q\ZZ$, which identifies with the grid
(or subgroup) $\mathbb{G}_q:=\{ k/q; k=0, 1, \cdots, q-1\}$
contained in the torus. We do not assume that $q$ is a prime
number at this point. We still denote by $e(x):=e^{2\pi i x/q}$
the exponential function adapted to the group $\ZZ_q$ and by $e_h$
the function $e(hx)$. Again the set
\begin{equation}\label{eq:idempotents_q}
\PP_q:=\left\{ \sum_{h\in H}e_h ~:~ H\subset \{0,\cdots, q-1\}
\right\}
\end{equation}
is called the set of \emph{idempotents} on $\ZZ_q$. In this
context, the set of idempotents has $2^q$ elements.

We then adapt the definition of $p$-concentration to the setting
of $\ZZ_q$.

\begin{definition}\label{def:gammapsharp}
Let $p>0$. We say that there is uniform (in $q$) $p$-concentration
for $\ZZ_q$ if there exists a constant $\gamma>0$ so that for each
prime number $q$ one can find an idempotent $f\in\PP_q$ with
\begin{equation}\label{p-conc}
2|f(1)|^p \geq \gamma  \sum_{k=0}^{q-1}|f(k)|^p.
\end{equation}
Moreover, writing $\gamma^\sharp_p(q)$ for the maximum of all such
constants $\gamma$, we put
$$
\gamma_p^\sharp:=\liminf_{q\to \infty} \gamma^\sharp_p(q).
$$
Then $\gamma_p^\sharp$ is called the uniform
level of $p$-concentration.
\end{definition}

Here we can formulate a discrete analogue of the problem in
\cite{CRMany, Many}. {\sl Does $q$-uniform concentration fail for
$p=1$?}

The reader may note that in order to define $p$-concentration in
the setting of $\ZZ_q$, one should also look for $f$ that
satisfies \eqref {p-conc}, but with $f(a)$, for some arbitrary
$a\in \ZZ_q$, in the left hand side. This is easy when $q$ is
prime. Indeed, for $a=0$ the Dirac mass at $0$, which is an
idempotent, has the required property with constant $1$.
Otherwise, if $a\neq 0$ and $f$ satisfies \eqref {p-conc}, then
the function $g(x):=f(a^{-1}x)$ satisfies the same inequality, but
with $g(a)$ in the left hand side. Here $a^{-1}$ is the unique
inverse for the multiplication in $\ZZ_q$.  Clearly $g(a) = f(1)$,
and all other values taken by $f$ are taken by $g$ since
multiplication is one-to-one in $\ZZ_q$ for $q$ prime, so that the
right hand side is the same for $f$ and $g$.
\begin{remark}\label{not-prime} We can also replace $1$ by $a$ in
the left-hand side of  \eqref {p-conc} when $q$ is any integer,
but $a$ and $q$ are co-prime.
\end{remark}
\medskip

As we said,  $p$-concentration on $\ZZ_q$ plays a role in proofs
for $p$-concentration on the torus.  In order to solve the
$2$-concentration problem on the torus, D\'echamps-Gondim,
Piquard-Lust and Queff\'elec \cite{DPQ, DPQ2} have considered the
concentration problem on $\ZZ_q$, proving the precise value that
we already mentioned,
\begin{equation}\label{case2-q}
\gamma_2^\sharp=\sup_{0\leq x}\frac {2\sin^2 x}{\pi
x}=0.4613\cdots.
\end{equation}
Moreover, they obtained $\gamma_p^\sharp\geq 2
(\gamma_2^\sharp/2)^{p/2}$ for all $p>2$. The last assertion is an
easy consequence of the decrease of $\ell^p$ norms with $p$, and
we have, in general,
\begin{equation}\label{comparison}  \gamma_p^\sharp\geq 2 (\gamma^\sharp_{p'}/2)^{p/p'}\end{equation} for $p>p'$.

Let us also mention that they considered the same problem for the
class of  positive definite polynomials, that is
\begin{equation}\label{eq:pos-def}
\PP_q^+:=\left\{ \sum_{h\in H}a_he_h ~:~ a_h\geq 0,
h\in\{0,\cdots, q-1\} \right\}.
\end{equation}

We say that there is uniform $p$-concentration on $\ZZ_q$ for the
class of positive definite polynomials if there exists some
constant $\gamma$ such that \eqref{p-conc} holds for some $f\in
\PP^+_q$. We denote by $c_p^+$ the level of $p$-concentration  for
the class of positive definite polynomials, which is defined as
the maximum of all admissible constants in (5) (similarly to the
class of idempotents).

With these notations, it has been proved in \cite{DPQ} that
$c_2^+=1/2$. Since the class of positive definite polynomials is
stable by taking products, it follows that, for all even integers
$2k$,
$$\gamma_{2k}^\sharp\leq c_{2k}^+\leq 1/2.$$

It is easy to see that there is uniform $p$-concentration on
$\ZZ_q$ for all $p>1$, using Dirichlet kernels. This has been used
in our paper \cite{BR}, where the discrete problem under
consideration here has been largely studied, at least for $p$ an
even integer.

On the other hand, coming back to our main point, i.e. to the case
of $p=1$, and using the recent results of B. Green and S. Konyagin
\cite{GK}, we answer negatively in this case, which gives an
affirmative answer to the conjecture of \cite{Many} for finite
groups $\ZZ_q$.

All the results on $\ZZ_q$ summarize in the following theorem,
which gives an almost complete answer to the $p$-concentration
problem under consideration, except for the best constants, which
are not known for $p\neq 2$.

\begin{theorem}\label{th:concentration} For all $1<p<\infty$ we
have uniform $p$-concentration on $\ZZ_q$.  We have
$\gamma^\sharp_2$ given by \eqref{case2}, then
$0.495<\gamma^\sharp_4\leq 1/2$. For all $p>2$, we have
$\gamma^\sharp_{p}>0.483$. On the other hand for $p\leq 1$ we do
not have uniform  $p$-concentration.
\end{theorem}

Positive results are implicitly contained in \cite{BR}, where they
are used as tools for the problem of concentration on the torus.
As far as necessary upper bounds for $\gamma^\sharp_p$ are
considered, since the polynomials $f$ with positive coefficients
have their maximum at $0$, we have the trivial upper bound
$\gamma^\sharp_p\leq 2/3$. Moreover, for $p$ an even integer, we
have seen that $\gamma^\sharp_p\leq 1/2$. Let us remark that
\eqref{comparison} provides an improvement on the bound $2/3$
between two even integers. Indeed, for $p\leq 2k$, we have
$$\gamma^\sharp_p\leq 2^{1-p/k}.$$

\bigskip

In the next two sections, we will consider the case of $\ZZ_q$,
first for $p>1$, then for $p=1$. Then, in Section 4, we will come
back to the case $p=2$ on the torus and exploit the proof for
giving concentration results by means of the use of the grid
$\mathbb{G}_q$. In the last section, we prove Theorem \ref{th:L1}.

\bigskip

We tried to keep the notations for the constants the same as in
\cite{BR}, since we refer to the proofs there, and apologize for
sometimes these notations seem more complicated than they should
be.

\section{uniform $p$-concentration}\label{finite}

In this section, we will recall the situation on the group $\ZZ_q$
by transferring the results that have been obtained for the grid
$$\mathbb{G}_q:=\{ k/q; k=0, 1, \cdots, q-1\}$$ contained in $\TT$.
By a slight abuse of notation, let us still denote
\begin{equation}\label{eq:idem_q}
\PP_q:=\left\{ \sum_{h\in H}e_h ~:~ H\subset \{0,\cdots, q-1\}
\right\}
\end{equation}
the set of trigonometrical idempotents of degree less than $q$ on
$\TT$, with $e_h$ denoting the exponential $e_h(x):=e^{2\pi i hx}$
adapted to $\TT$. When restricted to $\mathbb{G}_q$ identified
with $\frac 1q\ZZ_q$, it coincides with the corresponding
idempotent (the coefficients are the same, but the exponential is
now adapted to $\ZZ_q$) on $\ZZ_q$. This is a one-to-one
correspondence between idempotents of $\ZZ_q$ and idempotents of
degree less than $q$, since these last ones are determined by
their values on $q$ points, and, in particular, on $\mathbb{G}_q$.
We will prefer to deal with ordinary trigonometrical polynomials,
and see $\ZZ_q$ as the grid $\mathbb{G}_q$.

Unless explicitly mentioned, we will only consider Taylor
polynomials, that is, trigonometrical polynomials with only non
negative frequencies.

We consider the following quantities, written in these new
notations, and identify them with the quantities defined for
$\ZZ_q$ in the introduction.
\begin{equation}\label{new-c} \gamma^\sharp_p:=
\liminf_{q\rightarrow \infty} \gamma^\sharp_p(q), \qquad
\gamma^\sharp_p(q):=\sup_{R\in \PP_q} \frac{2\left |R\left(\frac
1q\right) \right|^p}{ \sum_{k=0}^{q-1} \left|R\left(\frac
kq\right) \right|^p}.
\end{equation}
One can obtain a lower bound of $\gamma^\sharp_p$, with $p>1$, by
the only consideration of the Dirichlet kernels
\begin{equation}\label{eq:Dndef}
D_n(x):=\sum_{\nu=0}^{n-1} e(\nu x) = e^{\pi i(n-1)x}
\frac{\sin(\pi n x)}{\sin(\pi x)}.
\end{equation}
Here the constraint on the degree restricts us to $n<q$. Having
$n$ and $q$ tend to infinity with $n/q$ tending to $t$, we proved
in \cite{BR} (see Lemma 35) that

\begin{lemma}\label{l:majB}
For  $p>1$, we have the inequality
\begin{equation}\label{eq:majB}
2(\gamma^\sharp_p)^{-1}\leq \inf_{0<t<1/2}B(p, t),
\end{equation} where, for $\lambda>1$,
\begin{equation}\label{eq:Bdef}
B(\lambda,t):=\left(\frac{\pi t}{\sin \pi
t}\right)^{\lambda}\left(1+2\sum_{k=1}^{\infty}
\left|\frac{\sin\left(k\pi t\right)}{k\pi
t}\right|^{\lambda}\right).
\end{equation}
\end{lemma}

\medskip

It is clear that $B(\lambda,t)$ is bounded for $\lambda>1$, so
that $\gamma^{\sharp}_p>0$ and there is uniform $p$-concentration:
just take as a bound the value for $t=1/4$. Let us try to get more
precise estimates. The computation of $\inf_{0<t<1/2}B(\lambda,
t)$ can be executed explicitly for $\lambda=2$ and $\lambda= 4$.
In the first case we recognize in the sum the Fourier coefficients
of $\chi_{[-t/2,t/2]}$, whose $L^2$ norm is $\sqrt t$. So
\eqref{eq:majB} leads to the minimization of the function $\frac
{2\sin^2 t}{\pi t}$, and to the estimate $ \gamma^\sharp_2\geq
\sup_{0\leq t}\frac {2\sin^2 t}{\pi t}=0.4613\cdots. $ This is the
formula given by D\'{e}champs-Gondim, Lust-Piquard and
Queff\'{e}lec in \cite{DPQ}. We refer to them for the necessity of
the condition, for which they give a smart proof. For $\lambda=4$,
we recognize in the sum of \eqref{eq:Bdef} the Fourier
coefficients of the convolution product
$\chi_{[-t/2,t/2]}*\chi_{[-t/2,t/2]}$, whose $L^2$ norm is equal
to $(2t^3/3)^{1/2}$. Using Plancherel Formula we obtain that
\begin{equation}\label{p=4}
\gamma^\sharp_4\geq \max_{0<t<1/2} \frac{3 \left(\sin ^4 (\pi t)
\right)}{\pi^4 t^3 } > 0.495.
\end{equation}
For larger integer values of $\lambda$, the computations do not
seem to be easily handled. But we can prove that there exists a
uniform lower bound for $\gamma^\sharp_p$ when $p\geq 6$. To see
this, we need another lemma that can be found in \cite{BR}. Let us
first give new definitions, relative to positive definite
polynomials.

As for idempotents, by the same slight abuse of notation, let us
still denote
\begin{equation}\label{pos-def}
\PP_q^+:=\left\{ \sum_{h\in H}a_he_h ~:~ a_h\geq 0,\,
h\in\{0,\cdots, q-1\} \right\}.
\end{equation}
the set of trigonometrical polynomials with non negative
coefficients of degree less than $q$ on $\TT$, with $e_h$ denoting
the exponential adapted to $\TT$. Again, when restricted to
$\mathbb{G}_q$, it coincides with the corresponding positive
definite polynomial with non negative coefficients on $\ZZ_q$, and
this defines a one-to-one correspondence between positive definite
polynomials of $\ZZ_q$ and positive definite polynomials on $\TT$
of degree less than $q$. The constant $c_p^+$ can then be defined
by
\begin{equation}\label{new-c+} c_p^+:=
\liminf_{q\rightarrow \infty} c_p^+(q), \qquad
c_p^+(q):=\sup_{R\in \PP_q^+} \frac{2\left |R\left(\frac 1q\right)
\right|^p}{ \sum_{k=0}^{q-1} \left|R\left(\frac kq\right)
\right|^p}.
\end{equation}
It is much easier to find positive definite polynomials in
$\PP_q^+$ than idempotents. In particular, whenever $P$ is in
$\PP_q$, then, for each positive integer $L$ the polynomial $Q$,
which has degree less than $q$ and has the same values on
$\mathbb{G}_q$ as $P^L$, is in $\PP_q^+$. So we can  take as well
powers of Dirichlet kernels as polynomials $R$ in the right hand
side of \eqref{new-c+}. This leads to the following bounds, using
Lemma \ref{l:majB}.
\begin{eqnarray} 2(c_p^+)^{-1}&\leq&\inf_{L\geq
1}\;\inf_{0<t<1/2}B(Lp, t)\nonumber\\
&\leq & \inf_{\kappa>0}\limsup_{\lambda\mapsto \infty}B\left
(\lambda, \kappa\sqrt{6 /\lambda }\right)\label{bound-above}\\
&\leq &  4.13273.\nonumber\end{eqnarray} The two last estimates
may be found in \cite{BR}, see (55), and lead to
\begin{equation}\label{for+ uniform}
c_p^+> 0.483.
\end{equation}
The first one gives a non explicit bound for a fixed $p$:
\begin{equation}\label{p-fixed}
  c_p^+\geq 2\sup _{L\geq
1}\;\sup_{0<t<1/2}B(Lp, t)^{-1}.
\end{equation}

We prove now that we have the same estimates for $\gamma^\sharp_p$
when $p>2$.

\begin{theorem}\label{th:gammapsharpp2} We have $\gamma_p^{\sharp}
> 0.483$ uniformly for all $p>2$ .
\end{theorem}

This is a consequence of the following proposition, which is more
general than the corresponding results in \cite{BR}.

\begin{proposition}\label{random} Let $p>2$ and $c>0$, $\varepsilon>0$.
Then there exists  $q_0:=q_0(c, \varepsilon)$ such that, if
$q>q_0$ and $P:=\sum_0^{q-1}a_h e_h$  is a  polynomial of degree
less than $q$ that satisfies the two conditions
\begin{equation}\label{cond-c}
cq\max_h|a_h| \leq \sum |a_h|\leq c^{-1}|P(1/q)|,
\end{equation}
\begin{equation}
|P(1/q)|\geq  c \left( \sum_{k=0}^{q-1}|P(k/q)|^p \right)^{1/p},
\label{concentr}
\end{equation}
then there exists a polynomial $Q$ of degree less than $q$, whose
coefficients are either $a_h/|a_h|$ or $0$, such that
 \begin{eqnarray}\label{at-p}
   |Q(1/q)|&\geq & (1-\varepsilon) |P(1/q)|,\\
   \left(\sum_{k=0}^{q-1}|Q(k/q)-P(k/q)|^p\right)^{1/p} & \leq &
   \varepsilon |P(1/q)|.\label{in-mean}
  \end{eqnarray}
\end{proposition}

Observe that, for $P$ positive definite, $Q$ is an idempotent. In
this case, the first condition can be reduced to $P(0)\geq
cq\max_h|a_h|$. Indeed, the fact that $|P(1/q)|\geq c P(0)$
follows from the second one.

Let us take the proposition for granted, and use it in our
context.

\comment{We claim that this allows us to conclude for the bound
below
\begin{equation}\label{foruniform}
\gamma_p^\sharp > 0.483.
\end{equation}
}

\begin{proof}[Proof of Theorem \ref{th:gammapsharpp2}]
Let us take for $P$ a positive-definite polynomial of degree less
than $q$ for which
$$
\frac{2\left |P\left(\frac 1q\right) \right|^p}{ \sum_{k=0}^{q-1}
\left|P\left(\frac kq\right) \right|^p}\geq c_0 >0.483.
$$
We claim that there exists an idempotent $Q$ for which the same
ratio is bounded below by $c_0 C(\varepsilon)$, with
$C(\varepsilon)$ tending to $1$ when $\varepsilon$ tends to $0$.
Indeed, we can apply the proposition as soon as we have proved
that $P$ satisfies the condition \eqref{cond-c} (uniformly for $q$
large). We have seen that $P$ can be taken as the polynomial of
degree less than $q$, which coincides with $D_n^L$ on the grid
$\mathbb{G}_q$, for $n$ chosen in such a way that $n/q\approx
t=\kappa\sqrt{6/\lambda}$  is small enough so that we approach the
extremum in \eqref{bound-above}. Next, it is easy to see that
$P(0)=n^L$, while $|\hat P(k)|\leq Ln^{L-1}$. So we have
\eqref{cond-c} with  a very small constant $c$, but what is
important that it does not depend on $q$ tending to $\infty$ (for
fixed $\varepsilon$). To conclude the proof, we use the fact that,
by Minkowski's inequality, and using the assumption on $P$, we
have
\begin{align*} \left(\sum_{k=0}^{q-1}
\left|Q\left(\frac kq\right) \right|^p\right)^{1/p}&\leq&
\left(\sum_{k=0}^{q-1} \left|P\left(\frac kq\right)
\right|^p\right)^{1/p}+ \e
|P(1/q)|\\
&\leq & ((2/c_0)^{1/p}+\e)|P(1/q)|\\
&\leq & (1-\e)((2/c_0)^{1/p}+\e)|Q(1/q)|.
\end{align*}
The constant tends to $(2/c_0)^{1/p}$ when $\e$ tends to $0$,
which concludes the proof.
\end{proof}

The same method leads to
\begin{equation}\label{p-fixed-gamma}
\gamma^\sharp_p\geq 2\sup _{L\geq 1}\;\sup_{0<t<1/2}B(Lp, t)^{-1}.
\end{equation}

This finishes the proof of the part of Theorem
\ref{th:concentration} concerning $p>1$, except for the proof of
Proposition \ref{random}, which we do now. It relies on the
construction of random polynomials, which may have an independent
interest.

\begin{proof}[Proof of Proposition \ref{random}] Without loss of
generality we may assume that $\max_h |a_h|=1$. We put
$\alpha_k:=|a_k|$ and $\sigma:=\sum \alpha_k$, so that $0\leq
\alpha_k\leq 1$ and $ cq\leq \sigma\leq c^{-1} |P(1/q)|$. We take
a sequence of independent random variables $X_0,X_1,\dots,X_{q-1}$
that follow the Bernoulli law with parameters $\alpha_0, \alpha_1,
\dots, \alpha_{q-1}$ on some probability space $(\Omega,
\mathcal{A}, \mathbb{P})$ and set
$$
P_\omega:=\sum_0^{q-1}b_h X_h(\omega)e_h
$$
with $b_h:=a_h/|a_h|$ for $a_h\neq 0$, otherwise $b_h=0$. Then the
expectation of $P_\omega$ is equal to $P$. We will prove that
$Q=P_\omega$ satisfies \eqref{at-p} and \eqref{in-mean} with
positive probability. Let us first consider \eqref{at-p}, and
prove that the converse inequality holds with probability less
than $1/3$ for $q$ large enough. Indeed, one has the inclusions
$$
\{\omega; |P_\omega(1/q)|\leq (1-\varepsilon)|P(1/q)|\}\subset \{\omega;
|P_\omega(1/q)-P(1/q)|>\varepsilon|P(1/q)|\},
$$
so that, by Markov inequality, using the fact that the variance of
$P_\omega (1/q)$ is $\sum \alpha_k(1-\alpha_k)\leq \sigma$, we
have
$$
\mathbb{P}\left(\left|\frac{P_\omega(1/q)}{P(1/q)}\right| \leq
1-\e\right)\leq c^{-2}\e^{-2}\sigma^{-1}.
$$
By \eqref{cond-c} we know that this quantity is small for $q$
large.

Next, to show \eqref{in-mean}, in view of \eqref{cond-c} it is
sufficient to prove that with probability $2/3$,
$$\sum_{k=0}^{q-1}|P_\omega(k/q)-P(k/q)|^p \leq c^{p}\varepsilon^p
\sigma^p.$$ We claim that there exists some uniform constant
$C_p$, for $p>2$, such that, for each $k$,
\begin{equation}\label{burkh}
\mathbb{E}(|P_\omega(k/q)-P(k/q)|^p)\leq C_p \sigma^{p/2}.
\end{equation}
Let us take this for granted and finish the proof. By simple
estimation
$$\mathbb{P}\left(\sum|P_\omega(k/q)-P(k/q)|^p \geq (c\varepsilon\sigma)^p
\right)\leq c^{-p}\varepsilon^{-p}C_p \,q\, \sigma^{-p/2}.$$ From
this we conclude easily, using the fact that $\sigma\geq cq$, so
that the right hand side tends to $0$ when $q$ tends to infinity.
Finally, \eqref{burkh} is a well-known property of independent
sums of Bernoulli variables, e.g. in \cite{BR} (Lemma 54) a proof
of the following lemma can be found.
\begin{lemma}\label{mart-bern}For $p>2$ there exists some constant $C_p$
with the following property. Let $\alpha_k\in [0,1]$ and $b_k\in
\CC$ be arbitrary for $k=0,1,\dots,N$. For $X_k$ a sequence of
independent Bernoulli random variables with parameter $\alpha_k$,
we have
$$
\mathbb{E}\left(|\sum_{k=0}^Nb_k (X_k-\alpha_k)|^{p}\right)\leq
C_p\cdot \max_{k=1,\dots,N} |b_k|^{p} \cdot (1+\sum_{k=0}^N
\alpha_k)^{p/2}.
$$
\end{lemma}

\end{proof}

Of course one would like to know whether constants are
 the same for classes $\PP_q$ and $\PP_q^+$. We know that it is
 not the case for $p=2$  thanks to
 the work of D\'{e}champs-Gondim, Lust-Piquard and Queff\'{e}lec, but the
 last proposition induces to conjecture that they are the same for
 $p>2$. Note that Proposition \ref{random} holds when \eqref{cond-c}  is replaced by the
 weaker assumption $\sigma\geq \delta(q)q^{2/p}\max |a_h|$, with $\delta$ tending to infinity with $q$.

\section{Failure of uniform $1$-concentration on $\ZZ_q$}\label{sec:proof}

We prove here the negative result of Theorem
\ref{th:concentration}. It will be more convenient, in this
section, to work directly on $\ZZ_q$, and not on the grid $
\mathbb{G}_q$. We now restrict to $q$ prime, which is sufficient
to conclude negatively.

 Assume that
there exists some constant $c$  and some idempotent $f=\sum_{h\in
H}e_h$ such that
\begin{equation}\label{1-conc}
|f(1)| \geq c  \sum_{k=0}^{q-1}|f(k)|.
\end{equation}
 We claim
that $H$ may be assumed having cardinality $\leq q/2$. Indeed, $H$
is certainly not the whole set $\{0, \cdots, q-1\}$, since the
corresponding idempotent is  $q$ times the Dirac mass at $0$.
Moreover, the idempotent $\widetilde{f}$, having spectrum $^c H$,
takes the same absolute values as $f$ outside $0$, while its value
at $0$ is $q-\card H$. So, if $\card H>q/2$, then $\widetilde{f}$
satisfies also \eqref{1-conc}.

From now on, let $r:= \card H \leq q/2$. We have by assumption
\eqref{1-conc} $\sum_{k=0}^{q-1}|f(k)|\leq |f(1)|/c \leq f(0)/c
=r/c$. So the function
$$
g:= r^{-1}\left (f-r\delta_0\right)
$$
 is $0$ at $0$, has $\ell^1$ norm bounded by $\frac 1c+1$, while
its Fourier coefficients are equal to $1/r-1/q$ ($r$ of them), or
$-1/q$, since the delta function has all Fourier coefficients
equal to $1/q$. But, according to Theorem 1.3 of \cite{GK}, we
should have $q\min_k |\hat g(k)|$ tending to $0$ when $q$ tends to
$\infty$ (note that the Fourier transform here is replaced by the
inverse Fourier transform in \cite{GK}, which is the reason for
multiplication by $q$ compared to the statement given there). This
gives a contradiction, and allows to conclude that there is no
uniform $1$-concentration. This finishes the proof.

\bigskip

We leave the following as an open question.
\begin{problem} In line with Definition \ref{def:gammapsharp},
for given fixed $q$ denote
$\gamma^{\sharp}_1(q):=\max\limits_{f\in \PP_q}
2|f(1)|/\sum_{k=0}^{q-1} |f(k)|$. Determine
$\beta:=\liminf\limits_{q\to\infty} \log (1/\gamma^{\sharp}_1(q))/
\log\log q$.
\end{problem}

Using the full strength of the result of \cite{GK}, the constant
$c$ in the proof of Theorem \ref{th:concentration} may be chosen
uniformly bounded from below in $q$ by $\log^{-\alpha}q$, with
$\alpha$ less than $1/3$ (that is, the proof by contradiction
shows that $c>\log^{-\alpha} q$ is not possible, hence $\beta \geq
1/3$).  On the other hand the Dirichlet kernel exhibits
$\gamma_1^{\sharp}(q)\geq C / \log q$, i.e. $\beta\leq 1$. This
leaves open the question if $\beta$ achieves 1, i.e.
$\log(1/\gamma_1^{\sharp}(q))/\log\log q$ can be taken anything
less than $1$. The problem is in relation with the Littlewood
conjecture on groups $\ZZ_q$, for which there has been new
improvements by Sanders \cite{S}.

\section{$2$-concentration on measurable sets}

We prove in this section that $\gamma_2\geq \gamma_2^\sharp$. The
converse inequality follows from the fact that the constant for
measurable sets is smaller than the one when restricted to open
sets, which is $\gamma^\sharp _2$, whose explicit value is given
by \eqref{case2-q}. In this paragraph we shall basically use the
method of Anderson et al. \cite{Many}. Our improvements are mainly
expository. The method is valid for all $p>1$, and we will write
it in this context, even if better results can be obtained for
$p\neq 2$. Indeed, it will be easier, later on, to explain how to
improve the method starting from this first one.

So we are going to prove the following proposition.

\begin{proposition}\label{p:comparison} For $p>1$, we have
$$\gamma_p\geq \gamma^\sharp _p.$$
\end{proposition}

\begin{proof}[Proof]
We are given an arbitrary symmetric measurable set, with $|E|>0$.
We want to find some idempotent $f$ that concentrates on $E$. We
will use a variant of Khintchine's Theorem in Diophantine
approximation, which we summarize in the next lemma (Proposition
36 in \cite{BR}).

\begin{lemma}\label{l:grid}Let $E$ be a measurable set
of positive measure in $\TT$. For all $\theta>0$, $\eta>0$ and
$q_0\in \mathbb{N}$, there exists an irreducible fraction $a/q$
such that $q>q_0$ and
\begin{equation}\label{khint}
\left|\left (\frac aq-\frac{\theta}{q^2},\frac aq
+\frac{\theta}{q^2}\right)\cap E\right|\geq (1-\eta)\frac{2\theta}
{q^2}.
\end{equation}
Moreover, given a positive integer $\nu$, it is possible to choose
$q$ such that $(\nu, q)=1$.
\end{lemma}

The parameter $\theta$ will play no role at the moment, so we can
set it as $1$. It will appear as necessary for generalizations
only later. We consider the grid $ \mathbb{G}_q:= \{ k/q; k=0, 1,
\cdots, q-1\}$ contained in the torus, for  $a$ and $q$ given by
Lemma \ref{l:grid}, for given values of $\eta$ and $q_0$ to be
fixed later on. We assume that $q$ is sufficiently large so that
we can find $R\in \PP_q$ with the property  that
\begin{equation}\label{given}
2|R(a/q)|^p\geq  c \sum_{k=0}^{q-1}|R(k/q)|^p,
\end{equation}
with $\e>0$ chosen arbitrarily small and $c>\gamma_p^\sharp-\e$.
When $a=1$, the existence of such a $P$ follows from the
definition of $\gamma_p^\sharp$. See Remark \ref{not-prime} for
the fact that we can replace $1$ by $a$ whenever $a$ and $q$ are
co-prime. We then claim that the polynomial $Q(t):=R(t) D_n(qt)$,
which is an idempotent, is such that
$$
\int_E |Q|^p \geq c \kappa(\varepsilon)\int_\TT |Q|^p,
$$
with $\kappa(\varepsilon)<1$ tending to $1$ when $\varepsilon$
tends to $0$, and parameters $\eta$ and $n$ are chosen suitably
depending on $\e$.

The idea of the proof goes as follows: since $D_n$ concentrates
the $L^p$ norm near $0$ (it can be concentrated in any subset $F$
of the interval $\left(-\frac{1}{q}, +\frac{1}{q}\right)$, with
$|F|>2(1-\eta)/q$), then $D_n(qt)$ concentrates equally on  the
$q$ subsets around the  points of the grid $\mathbb G_q$. We take
$F$ such as $qt$ belongs to $F$ when $t$ belongs to $\left(\frac
aq-\frac{\theta}{q^2},\frac aq +\frac{\theta}{q^2}\right)\cap E$.
Now multiplication by $R$ will concentrate the integral on the
subset around $a/q$, which we wanted. We need to know that the
polynomial $R$ is almost constant on each of these subsets, which
is given by Bernstein's Theorem.

Let us now enter into details. We have the following lemma on
Dirichlet kernels.

\begin{lemma}\label{dirichlet}
Let $p>1$. For $\e$ given, one can find $\eta>0$ and $\delta_0>0$
such that, for all  $0<\delta<\delta_0$ , if $F$ is a measurable
subset of $(-\delta, +\delta)\subset \TT$ of measure larger  than
$2\delta(1-\eta)$, we can find some suitable $n\in\NN$ so that
$$\int_F |D_n |^p\geq (1-\e)\int_\TT |D_n|^p.$$
\end{lemma}
\begin{proof}[Proof] It is well known that $\int_\TT |D_n|^p\geq
\kappa_p n^{p-1}$ (see \cite{Many} for instance for precise
estimates). So it is sufficient to prove that we can obtain
$$ \int_{^cF} |D_n |^p\leq \varepsilon n^{p-1}.$$
This is a consequence of the fact that $$\int_{(-\delta,
+\delta)\setminus F} |D_n|^p\leq 2n^p\eta \delta,$$ while
$$\int_{\TT\setminus(-\delta, +\delta)} |D_n|^p\leq \left(\frac \pi
2\right)^p\int_{|t|>\delta} t^{-p}dt= \kappa'_p \delta^{1-p}.$$ We
choose for $n$ the smallest integer larger than $
(2\kappa'_p/\e)^{1/(p-1)}\delta^{-1}$ and $\eta$ such that
$8(2\kappa'_p/\e)^{1/(p-1)}\eta=\e$.

We remark that here we did not need the flexibility linked to the
parameter $\delta_0$. It is here for further generalizations.
\end{proof}

Next we recall classical Bernstein and Marcinkiewicz-Zygmund type
inequalities, in the forms tailored to our needs and proved in
 \cite{BR}, Lemma 41. Recall that here polynomials are Taylor polynomials,
 that is, trigonometrical polynomials with only non negative
 frequencies, which is the case for the polynomial $R$.
\begin{lemma}\label{l:Bernstein}
For $1< p<\infty$ there exists a constant $K_p$ such that, for $P$
a polynomial of degree less than $q$ and for $|t|<1/2$, we have
the two inequalities
\begin{equation}\label{maj-grid}
\sum_{k=0}^{q-1}|P( t+k/q)|^p\leq K_p \sum_{k=0}^{q-1}|P(k/q)|^p,
\end{equation}
\begin{equation}\label{maj-grid1}
\sum_{k=0}^{q-1}\left||P(t+k/q)|^p-|P(k/q)|^p \right| \leq K_p
|qt| \sum_{k=0}^{q-1}|P(k/q )|^p.
\end{equation}
\end{lemma}
For our polynomial $R$, this gives  the inequality
\begin{equation}\label{bernstein-a}
||R(t)|^p-|R(a/q)|^p|\leq 2c^{-1}K_p qt |R(a/q)|^p,
\end{equation}
This implies that, for $|t-\frac aq|<\frac \theta {q^2}$ with $q$
large enough,
 \begin{equation}\label{inequality-1}
|R(t)|^p\geq (1-\e)|R(a/q)|^p.
\end{equation}
We have also, for $|t|<\frac\theta {q^2}$, that
$$
\sum_{k=0}^{q-1}|R(t+k/q)|^p \leq \sum_{k=0}^{q-1}|R(k/q)|^p
+2K_p\frac \theta q c^{-1}|R(a/q)|^p
$$
which leads to the inequality, valid for $|t|<\frac\theta {q^2}$
for $q$ large enough,
\begin{equation}\label{inequality-2}
\sum_{k=0}^{q-1}|R(t+k/q)|^p \leq 2c^{-1}(1+\e)|R(a/q)|^p.
\end{equation}
Let us finally remark that \eqref{maj-grid} leads to the
following, valid for all $t$.
\begin{equation}\label{inequality-3} \sum_{k=0}^{q-1}|R(t+k/q)|^p
\leq  2 c^{-1}K_p|R(a/q)|.
\end{equation}

We can now proceed to the proof of the required inequality for
$R$. We have fixed $\e$ and chosen $q_0$ large enough so that
estimates \eqref{inequality-1} and \eqref{inequality-2} hold
(recall that for the moment $\theta=1$). Then we use Lemma
\ref{l:grid}, which fixes some $a/q$, and find $D_n$, which is
assumed to be adapted to $\delta:=\frac{\theta}q$. We denote
$\tau^p:=\int_{\TT}|D_n|^p$ and $I:= \left(\frac
aq-\frac{\theta}{q^2},\frac aq +\frac{\theta}{q^2}\right)$.
\begin{align}\notag
\frac 12 \int_E |Q|^p \geq \int_{I\cap E} |R|^p|D_n|^p & \geq (1-
\varepsilon)|R(a/q)|^p \int_{I\cap E} |D_n(qt)|^pdt \\
& \geq \frac{1}{q}(1-\varepsilon)|R(a/q)|^p
~~ \int_{F\cap (-\delta, +\delta)} |D_n|^p\notag \\
& \geq \frac{(1-\varepsilon)^2\tau^p}{q}|R(a/q)|^p. \label{ineq-3}
\end{align}
Here $F$ is the pre-image  by $t\mapsto qt$ of $I\cap E$, which
has measure at least $2(1-\eta)\delta$, and so concentrates the
integral of $|D_n|^p$.

Let us now look for a bound of the whole integral. We write
$$\int_\TT |Q|^p = \int_{-1/q}^{1/q}\left(\sum_k|R(t+\frac
kq)|^p\right)|D_n(qt)|^p dt$$ and cut the integral into two parts,
depending on the fact that $|t|\leq \frac \theta{q^2}$ or not. For
the first part we use \eqref{inequality-2},  for the second one
\eqref{inequality-3}. We recall that the integral of $D_n$ outside
the interval $(-\theta/q, \theta/q)$ is bounded by $\e \tau^p$.
Finally
\begin{eqnarray*}
 \int_\TT |Q|^p
& \leq &2c^{-1}\,\frac{1+\varepsilon}{q}|R(a/q)|^p
~~ \int_{\TT} |D_n|^p +2c^{-1} K_p\,\cdot \,\frac{\varepsilon}{q} |R(a/q)|^p \int_{\TT} |D_n|^p \\
& \leq &2c^{-1}\,\frac{(1+C\varepsilon)\tau^p}{q}|R(a/q)|^p.
\end{eqnarray*}
We conclude by comparison with \eqref{ineq-3}.
\end{proof}

As said above, we have obtained optimal results for $p=2$. At this
point, we can see how results can be improved  for $p\neq 2$. The
main point is the possibility to replace the Dirichlet kernel
$D_n$ by an idempotent $T$, which satisfies nearly the same
properties as the Dirichlet kernel that are summarized in Lemma
\ref{dirichlet}, but has the additional property to have
arbitrarily large gaps. More precisely, we say that {\sl $T$  has
gaps larger than $N$ } if $|k-k'|\leq N$ implies that one of the
two Fourier coefficients $\hat T(k)$ and $\hat T(k')$ is zero. We
state the existence of such idempotents $T$ as a lemma, and refer
to \cite{BR} for their construction.

\begin{lemma}\label{l:peak-meas}
Let $p>0$ different from $2$. Then for $\varepsilon>0$ there
exists $\delta_0>0$ and $\eta>0$ such that, for all
$\delta<\delta_0$ and $N\in \mathbb{N}$,  if $E$ is a measurable
set  that satisfies, for $\alpha=0$, the assumption $|E\cap [
\alpha-\delta, \alpha+\delta]|>2(1-\eta)\delta$, then there exists
an idempotent $T$ with gaps larger than $N$ such that
$$\int_{E\cap
[\alpha-\delta, \alpha+\delta]} |T|^p>(1-\varepsilon)\int_0^1
|T|^p.$$ Moreover, if $p$ is not an even integer, this is also
valid for $\alpha=1/2$.
\end{lemma}

For the moment we use this lemma with $\alpha=0$. We are no more
restricted to consider polynomials of degree less than $q$ in
order that $R(t)T(qt)$ be an idempotent. It is sufficient that the
degree of $R$ be less than $Nq$, and, since $N$ is arbitrary, this
gives essentially no constraint. The fact that $R$ has degree less
than $q$ was also used for \eqref{inequality-1} and
\eqref{inequality-2}. It is where the flexibility given by  the
parameter $\theta$ can be used: if $R$ has degree less than $q^2$,
then roughly speaking we can also use Bernstein Inequality, but
$\theta/q$ has to be replaced by $\theta$ in \eqref{bernstein-a}.
This is of no inconvenience, since $\theta$ can be chosen
arbitrarily small.

At this point, we could proceed with a polynomial of degree less
than $q^2$ for \eqref{inequality-1}, but certainly not for Lemma
\ref{l:Bernstein}, since such a polynomial can be identically $0$
on the grid $\mathbb{G}_q$. To develop such inequalities for
polynomials $S$ of degree larger than $q$, we will restrict to
those that can be written as $S(t):=R(t)R((q+1)t)$, with $R$ an
idempotent that satisfies \eqref{given}, but for $2p$ instead of
$p$ (so that the condition on $p$ is now $p>1/2$). The important
point is that $S$ is also an idempotent, and so is $ST$ if $T$ has
sufficiently large gaps. Also $|S(k/q)|^p =|R(k/q)|^{2p}$ at each
point of the grid, and in particular at $a/q$. Moreover, it is
easy to see that, for $\theta$ small enough, one still has the
inequalities \eqref{inequality-1}, \eqref{inequality-2} and
\eqref{inequality-3} with $2p$ in place of $p$, both for the
polynomials $R(t)$ and $R((q+1)t)$ (for this last one we have to
choose $\theta$ small enough, as we mentioned earlier.) The fact
that \eqref{inequality-1}, \eqref{inequality-2} and
\eqref{inequality-3} are valid for $S$ follows from Cauchy-Schwarz
Inequality. The rest of the proof goes the same way as the
previous one and leads to the following, for which we leave
details to the reader.

\begin{proposition} \label{2pgivesp} One has $p$-concentration for $p>1/2$, and,
for $p\neq 2$, one has the inequality $\gamma_p\geq
\gamma_{2p}^\sharp$. In particular $\gamma_1\geq
\gamma_2^{\sharp}$.
\end{proposition}

We could as well have taken $S=R_1R_2$ and used H\"older's
Inequality, taking $R_1$  approaching the maximum concentration on
the grid for the exponent $r$ and $R_2$  approaching the maximum
concentration on the grid for the exponent $s$, with $\frac
pr+\frac ps=1$. This leads to the following generalization of the
last proposition.

\begin{proposition} \label{with-holder} One has $p$-concentration for $p>1/2$, and,
for $p\neq 2$, one has the inequality $\gamma_p\geq
\left(\gamma_{r}^\sharp\right)^{p/r}\left(\gamma_{s}^\sharp\right)^{p/s}$
for all $r>p$ and $s>p$ such that $\frac pr+\frac ps=1$.
\end{proposition}

Before concluding this section, let us make a last observation.
Once we use an idempotent $T$ with arbitrarily large gaps, it is
not difficult to build idempotents with arbitrarily large gaps. It
is sufficient to start from the polynomial $R(\nu t)$, with $\nu$
arbitrarily large. Recall  that when using Lemma \ref{l:grid}, we
can take  $q$ such that $(\nu,q)=1$. This means that there exists
$b$ (mod $q$) such that $\nu a=b$ (mod $q$), and we choose $R$
that satisfies \eqref{given}, but with $b/q$ in place of $a/q$.
The rest of the proof can be adapted. We state it as a
proposition.

\begin{proposition} In Proposition \ref{p:comparison} and Proposition
\ref{with-holder}, when $p\neq 2$, we can have arbitrarily large
gaps. That is, when $1/2<p\neq 2$, given a symmetric measurable
set $E$ of positive measure, and any constant $c<\gamma_p^\sharp$
(resp.
$\left(\gamma_{r}^\sharp\right)^{p/r}\left(\gamma_{s}^\sharp\right)^{p/s}$),
there exists an idempotent $P$ with arbitrarily large gaps such
that
$$
\int_E |P|^p>c \int_{\TT} |P|^p.
$$
\end{proposition}

\section{ Improvement of constants for $p$ not an even integer }

We proved in \cite{BR} that $\gamma_p=1$ for $p>1$ and $p$ not an
even integer. Let us give the main lines of the proof, which will
be used again for the improvement of the constant when $p=1$. As
we shall see, it has been slightly simplified compared to the
proof in \cite{BR}. The main ingredient is the fact that there are
idempotents that concentrate as the Dirichlet kernels, but with
arbitrarily large gaps, and at $1/2$ instead of $0$. We have
already stated this in Lemma \ref{l:peak-meas}.

\medskip

If we take such a peaking function $T$, then $T(qx)$ concentrates
around the points of the translated grid
\begin{equation}\label{def:gridhalf} \mathbb{G}_q^\star:=\frac 1{2q}
+  \mathbb{G}_q = \left\{\frac {2k+1}{2q}\ \ ;\ \ k=0, \cdots,
q-1\right\}.
\end{equation}
We have considerably gained with this new grid compared to
$\mathbb{G}_q$ because $0$ -- where, by positive definiteness, we
always must have a maximal value of any idempotent -- does not
belong to the grid any more, and thus we will even be able to find
idempotents $P$ such that the maximal value of $|P|$ (over the
grid) will be attained at the points $\pm 1/2q$, moreover, the sum
of the values $|P|^p$ on $\mathbb{G}_q^\star$ is just slightly
larger than $2|P(1/2q)|^p$.

Let us interpret the new constants that we will introduce in terms
of another concentration problem on a finite group. More
precisely, we view $\mathbb{G}_q^\star$ as
$\mathbb{G}_{2q}\setminus \mathbb{G}_q$, and identify
$\mathbb{G}_{2q}$ with $ \ZZ_{2q}$, while $ \mathbb{G}_{q}^*$
identifies with a coset. Recall that the idempotents on $
\ZZ_{2q}$ are identified with polynomials in $\PP_{2q}$. We are
interested in relative concentration inside the coset, and give
the following definition.

\begin{definition}\label{def:Gammap-star} We define
\begin{equation}
\Gamma_p^\star:= \sup_{K<\infty} \liminf_{q\rightarrow \infty}
\Gamma_p^\star(q,K), \end{equation} where $\Gamma_p^\star(q,K)$ is
the maximum of all constants $\gamma$ for which there exists $R\in
\PP_{2q}$ satisfying
\begin{eqnarray}
2\left|R\left(\frac 1{2q}\right) \right|^p&\geq & \gamma
\sum_{k=0}^{q-1} \left|R\left(\frac {2k+1}{2q}\right) \right|^p  \label{Gam-star}  \\
2\left|R\left(\frac 1{2q}\right) \right|^p&\geq & \gamma K^{-1}
\sum_{k=0}^{q-1}\left|R\left(\frac{k}{q}\right)\right|^p.
\label{cond-K}
\end{eqnarray}
\end{definition}

In other words, $\Gamma_p^\star$ is positive when there  is
uniform concentration at $1/2q$, (which is the case for $p>1$),
but the grids $\mathbb{G}_q$ and $\mathbb{G}_q^\star$ do not play
the same role; the constant $\Gamma_p^{\star}$ is only the
relative concentration on $\mathbb{G}_q^{\star}$, which we try to
maximize.

\begin{remark}\label{not-prime-bis} We can also replace $1$ by $2a+1$ in the left-hand side of
 \eqref {Gam-star} when $q$ is any integer, but $2a+1$ and $2q$
co-primes.
\end{remark}
This is the equivalent of Remark \ref{not-prime}. Multiplication
 by  $b$, such that $b(2a+1)\equiv 1$ modulo $2q$, will send $1$ to
$2a+1$ and define a bijection on $\mathbb{G}_q^\star$ (resp.
$\mathbb{G}_q$).
\medskip

Lower bounds for $\Gamma_{p}^\star $ are given in the lemma below,
which is a slight modification of Lemma 34 in \cite{BR}.
\begin{lemma}\label{l:majA}
For  $p>1$, we have the inequality
\begin{equation}\label{eq:majA}
\frac1{\Gamma^\star_p} \leq \inf_{0<t<1/2}A(p, t),
\end{equation}
 where, for $\lambda>1$,
\begin{equation}\label{eq:Alimit}
A(\lambda, t):=\frac 1 {\left(\sin(\pi t) \right)^{\lambda}}
\sum_{k=0}^{\infty}\left|\frac{\sin\left((2k+1)\pi t\right)}
{2k+1}\right|^{\lambda}.
\end{equation}
\end{lemma}
The inequality is obtained by taking  Dirichlet kernels $D_n$,
with $n/2q$ tending to $t$, a point that will be used later on.
Observe that $A(\lambda, t)$ tends to $\infty$ when $t$ tends to
$0$, so that the infimum is obtained away from $0$. The uniformity
in the second inequality \eqref{cond-K} is given by a bound of (a
small modification of) $B(\lambda, t)$ defined in \eqref{eq:Bdef},
for which we have the inequality
\begin{equation}\label{K-above}
B(\lambda, t)\leq \left (\frac \pi 2\right)^\lambda +2 \left
(\sum_k k^{-\lambda}\right)t^{-\lambda}.
\end{equation}
$$ $$

Observe that (for fixed $t$) $A(\lambda,t)$, and hence also
$\inf_{0<t<1/2} A(\lambda,t)$ are decreasing functions of
$\lambda$. In \cite{BR} recognizing the Fourier coefficients (at
$k$ and $-k$) of the function $\frac {\pi}2 \left
(\chi_{[-t/2,t/2]}(x)-\chi_{[-t/2,t/2]}(x-1/2)\right)$ we used
Plancherel Formula to calculate
\begin{equation}\label{exact}
A(2,t)=\frac {\pi^2 t}{4\sin^2(\pi t)}.
\end{equation}
Substituting $x=\pi t$ and recalling \eqref{case2} we find that
$$
\Gamma^\star_2\geq 2\gamma_2 \approx 0.9226.
$$
Moreover, it is easy to see that $ \inf_{0<t<1/2} A(\lambda,t)$ is
left continuous in $\lambda$ at $2$, so that
\begin{equation}\label{eq:c2below}
\liminf_{p\to 2-0} \Gamma_p^\star \geq 2\gamma_2.
\end{equation}

Our main estimate for $\Gamma_p^\star$ is the following.
\begin{proposition}\label{optimal}
For $p>2$ we have $\Gamma_p^\star = 1$.
\end{proposition}

We postpone the proof of this proposition and show how to use it.
We need  an adaptation of the Khintchine 's type theorem that we
used in the last section. The next lemma uses the inhomogeneous
extension of Khintchine's Diophantine approximation theorem, first
proved by Sz\"usz \cite{Sz} and later generalized by Schmidt
\cite{Sch}. This is Proposition 37 of \cite{BR}.

\begin{lemma}\label{l:grid-half}Let $E$ be a measurable set
of positive measure in $\TT$. For all $\theta>0$, $\eta>0$ and
$q_0\in \mathbb{N}$, there exists an irreducible fraction
$(2k+1)/(2q)$ such that $q>q_0$ and
\begin{equation}\label{szusz}
\left|\left [\frac {2k+1}{2q}-\frac{\theta}{q^2},\frac {2k+1}{2q}
+\frac{\theta}{q^2}\right]\cap E\right|\geq (1-\eta)\frac{2\theta}
{q^2}.
\end{equation}
Moreover, given a positive integer $\nu$, it is possible to choose
$q$ such that $(\nu, q)=1$.
\end{lemma}

Our main result is the following.

\begin{theorem}\label{Star}
For $p$ not an even integer, one has the inequalities
$\gamma_p\geq \Gamma_{p}^\star $ and $\gamma_p\geq
\left(\Gamma_{r}^\star\right)^{p/r}\left(\Gamma_{s}^\star\right)^{p/s}$
for all $r>p$ and $s>p$ such that $\frac pr+\frac ps=1$. Moreover,
given a symmetric measurable set $E$ of positive measure, and any
constant $c<\Gamma_p^\star$ (resp.
$\left(\Gamma_{r}^\star\right)^{p/r}\left(\Gamma_{s}^\star\right)^{p/s}$),
there exists an idempotent $P$ with arbitrarily large gaps such
that
$$
\int_E |P|^p>c \int_{\TT} |P|^p.
$$
\end{theorem}
\begin{proof}[Proof]
We shall first  prove the inequality $\gamma_p\geq
\Gamma_{p}^\star $. We will then show how to modify the proof for
the other statements.

We are given a symmetric measurable set $E$. We consider the grid
$ \mathbb{G}_q^\star = \mathbb{G}_{2q}\setminus \mathbb{G}_q$
contained in the torus, with $a$ and $q$ given by Lemma
\ref{l:grid-half}. At this point we have already fixed some
$\e>0$. The values of $q_0$, $\eta$ and $\theta$ are also fixed,
but we will say how to choose them later on. We assume that $q$ is
sufficiently large so that we can find $R\in \PP_{2q}$ with the
property that
\begin{equation}\label{given-star}
  2|R(\frac 1{2q}+\frac aq)|^p\geq  c
\sum_{k=0}^{q-1}|R(\frac 1{2q}+\frac kq)|^p,
\end{equation}
 with $c>(1-\e)\Gamma_p^\star$. Moreover we can assume that
\begin{equation}\label{grid-star}
\sum_{k=0}^{q-1}|R(\frac kq)|^p\leq 2Kc^{-1}|R(\frac 1{2q}+\frac
aq)|^p
\end{equation}
for some uniform constant $K$. The existence of such an $R$ is
given by Definition  \ref{def:Gammap-star} and by the remark just
after. Once chosen $R$, we choose a peaking function $T$ at $1/2$
for the value $\e$. We assume now that $\eta$ has been chosen
sufficiently small for the existence of such a function $T$, built
for $\delta:=\theta/q^2$, which is possible if $\theta
q_0^{-2}\leq \delta_0$.

We choose the idempotent $Q(t):=R(t)T(qt)$ (indeed it is an
idempotent if $T$ has sufficiently large gaps) and fix  $I:=
\left(\frac {2a+1}{2q}-\frac{\theta}{q^2},\frac {2a+1}{2q}
+\frac{\theta}{q^2}\right)$. We also put
$\tau^p:=\int_{\TT}|T|^p$. From this point on, the proof follows
the same lines as the proof of Proposition \ref{p:comparison}. We
have the inequality
\begin{align*}
\frac 12 \int_E |Q|^p \geq \int_{I\cap E} |R|^p|T|^p & \geq (1-
\varepsilon)|R((2a+1)/(2q))|^p \int_{I\cap E} |T(qt)|^pdt \\
& \geq \frac{1}{q}(1-\varepsilon)|R((2a+1)/(2q))|^p
~~ \int_{F\cap (-\delta, +\delta)} |T|^p \\
& \geq \frac{(1-\varepsilon)^2\tau^p}{q}|R((2a+1)/(2q))|^p.
\end{align*}
We have used that the pre-image $F$ of $I\cap E$ by $t\mapsto qt$
has measure at least $2(1-\eta)\delta$, and concentrates the
integral of $|T|^p$ at $1/2$. We have also used the inequality,
\begin{equation}\label{inequ-1}
|R(t)|^p\geq (1-\e)|R(\frac {2a+1}{2q})|^p,
\end{equation}
valid for $|t-\frac {2a+1}{2q}|<\frac \theta {q}$ with $\theta$
small enough. This is an easy consequence of Lemma
\ref{l:peak-meas} for polynomials of degree $2q$, since the sum of
values of $|R|^p$ on the whole grid $\mathbb{G}_{2q}$ is bounded
by $2c^{-1}(K+1) $ times its value at $(2a+1)/(2q)$. Just take
$\theta$ small enough (we fix $\theta$ in such a way that this is
valid).

Before going on, let us remark that the other two basic
inequalities can be deduced from Lemma \ref{l:peak-meas}. First,
for $|t-\frac {2a+1}{2q}|<\frac \theta {q}$ with $\theta$ small
enough, we have also
\begin{equation}\label{inequ-2}
\sum_{k=0}^{q-1}|R(t+\frac {k}{q})|^p \leq 2 c^{-1}(1+\e)|R(\frac
{2a+1}{2q})|^p.
\end{equation}
Finally, for all $t$, we have, for some constant $\kappa$,
\begin{equation}\label{inequ-3} \sum_{k=0}^{q-1}|R(t+\frac kq)|^p
\leq \kappa|R(\frac {2a+1}{2q})|^p.
\end{equation}
Here we can take $\kappa:=2c^{-1} K_p(K+1)$. Next we  look for a
bound of the whole integral
$$\int_\TT |Q|^p = \int_{0}^{1/q}\left(\sum_k|R(t+\frac
kq)|^p\right)|T(qt)|^p dt$$ and cut the integral into two parts,
depending on the fact that $|t-\frac 1{2q}|\leq \frac{\theta}{q}$
or not. For the first part we use \eqref{inequ-2},  for the second
one \eqref{inequ-3}. We recall that the integral of $T$ outside
the interval $(\frac 12-\frac \theta q, \frac 12+\frac \theta q)$
is bounded by $\e \tau^p$.
\begin{eqnarray*}
 \int_\TT |Q|^p
& \leq & 2c^{-1}\,\frac{1+\varepsilon}{q}|R(\frac {2a+1}{2q})|^p
~~ \tau^p
+\kappa\frac{\varepsilon}{q} |R(\frac {2a+1}{2q})|^p \tau^p \\
& \leq & 2c^{-1}\,\frac{(1+C\varepsilon)\tau^p}{q}|R(\frac
{2a+1}{2q})|^p.
\end{eqnarray*}
We conclude by comparison with the integral on $E$. This allows to
conclude for the first case, $\gamma_p\geq \Gamma_{p}^\star $.
\medskip

Let us now indicate the necessary modification for finding
$\gamma_p\geq \left(\Gamma_{r}^\star\right)^{p/r}
\left(\Gamma_{s}^\star\right)^{p/s}$. In the following we denote
$r_1:=r$ and $r_2:=s$: the index $j$ will always cover the two
values $j=1$ and $j=2$. Instead of starting from one polynomial,
we start from two polynomials $R_1$ and $R_2$ in $ \PP_{2q}$,
which satisfy the following inequalities, for $j=1,2$.
\begin{equation}\label{given-star2}
2|R_j(\frac {2a+1}{2q})|^{r_j}\geq  c_j \sum_{k=0}^{q-1}|R_j(\frac
{2a+1}{2q})|^{r_j},
\end{equation}
with $c_j>(1-\e)\Gamma_{r_j}^\star$. Moreover we assume that
\begin{equation}\label{grid-star2}
\sum_{k=0}^{q-1}|R_j(\frac kq)|^{r_j}\leq
2Kc^{-1}|R_j(\frac{2a+1}{2q})|^{r_j}
\end{equation}
for some uniform constant $K$. We then put
$R(t):=R_1(t)R_2((2q+1)t)$. We remark that, on $\mathbb G_{2q}$,
the values of $R$ coincide with the values of the product $R_1
R_2$. We will prove that we still have inequalities
\eqref{inequ-1} and \eqref{inequ-2} for $|t-\frac
{2a+1}{2q}|<\frac \theta {q^2}$, and \eqref{inequ-3} for all $t$.
Let us first prove that \eqref{inequ-3} holds for some constant
$\kappa$. Indeed, by H\"older Inequality with conjugate exponents
$r_1/p$ and $r_2/p$ and periodicity of $R_2$, we have
$$
\sum_{k=0}^{q-1}|R(t+\frac kq)|^p\leq
\left(\sum_{k=0}^{q-1}|R_1(t+\frac kq)|^{r_1}\right)^{\frac
p{r_1}} \times \left(\sum_{k=0}^{q-1}|R_2((2q+1)t+\frac
kq)|^{r_2}\right)^{\frac p{r_2}}.
$$
Both factors are bounded, up to a constant, respectively by
$|R_1(\frac{2a+1}{2q})|^p$ and $|R_2(\frac{2a+1}{2q})|^p$, which
allows to conclude.

 In view of \eqref{inequ-1} and \eqref{inequ-2}, we remark that,
 when $t$ differs from $\frac {2a+1}{2q}$ by less than $\frac
\theta {q^2}$, then $(2q+1)t$ differs from $\frac {2a+1}{2q}$
(modulo $1$) by less than $\frac {3\theta} {q}$. So we still have,
for  $|t-\frac {2a+1}{2q}|<\frac \theta {q^2}$ with $\theta$ small
enough,
 \begin{equation}\label{inequ-4}
|R(t)|^p\geq (1-\e)|R(\frac {2a+1}{2q})|^p.
\end{equation}
For Inequality \eqref{inequ-2}, we first use H\"older Inequality
with conjugate exponents $r_1/p$ and $r_2/p$ as before, then the
same kind of estimate for each factor.

From this point, the proof is the same.

It remains to indicate how to modify the proof to get peaking
idempotents with arbitrarily large gaps.  So we fix $\nu$ as a
large odd integer, and we will prove that we can replace  the
polynomial $R$ used above  by some
$$
S(x):=R_1(\nu x) R_2((2q+1)\nu x),
$$ which has gaps larger than $\nu$. Recall first
that we can take arbitrarily large $q$ satisfying $(\nu,q)=1$, and
get an idempotent by multiplication by $T(qx)$ for $T$ having
sufficiently large gaps. The value taken by the polynomial $S$ at
$\frac{2a+1}{2q}$ is the value of $R_1R_2$ at $\frac{2b+1}{2q}$,
with $\nu (2a+1)\equiv 2b+1 $ mod $2q$. So we choose $R_1$ and
$R_2$ as before, but with $b$ in place of $a$.

From this point the proof is identical, apart from an additional
factor $\nu$, which modifies the value of $\theta$. We know that
$S(\nu x)$ and $R(x)$take globally the same values on both grids
$\GG$ and $\GS$, because in each case we multiply by an odd
integer that is coprime with $2q$.
\end{proof}

Now Theorem \ref{th:L1} is an easy consequence of Proposition
\ref{optimal} and Theorem \ref{Star}: take $r<2$ and $s>2$, so
that $\gamma_1\geq 1 \cdot (\Gamma_{r}^\star)^{1/r}$, and take the
limit of $\Gamma_r$ for $r\to 2-0$ using \eqref{eq:c2below}.

\bigskip

\begin{proof}[Proof of Proposition \ref{optimal}]
The proof is in the same spirit as the proof of the inequality
$\gamma_p^\sharp > 0.483$. Let us first fix $c<1$ and prove that
we can find a positive definite polynomial of degree less than
$2q$ such that
 $$2|P(\frac 1{2q}+\frac aq)|^p\geq  c
\sum_{k=0}^{q-1}|P(\frac 1{2q}+\frac kq)|^p,$$ while $$2|P(\frac
1{2q}+\frac aq)|^p\geq  c \sum_{k=0}^{q-1}|P(\frac kq)|^p.$$
Indeed, it is proved in \cite{BR} (and elementary) that $A(Lp,
1/4)$ has limit $1/2$ when $L$ tends to $\infty$, which means that
we can take for $P$ a polynomial that coincides with $D_n^L$ on
the grid $\mathbb{G}_{2q}$. We fix $L$  large enough, and choose
$n$ to be approximately $q/4$.  The second inequality follows from
\eqref{K-above}.

At this point one can use Proposition \ref{random}, with $q$
replaced by $2q$, to find the idempotent $Q$.

\end{proof}

\end{document}